\documentclass[12pt, oneside]{amsart}
\usepackage{amsmath,amsfonts,amssymb, tikz, amsthm}
\usepackage[utf8]{inputenc}
\usepackage[english]{babel}
\usepackage{xcolor}
\usepackage[tableposition=top]{caption} 
\newtheorem*{theorem*}{Theorem}
\newtheorem{theorem}{Theorem}
\newtheorem{lemma}{Lemma}
\newtheorem{proposition}{Proposition}

\newtheorem*{Conjecture}{Conjecture}

\theoremstyle{remark}

\numberwithin{equation}{section}
\numberwithin{corollary}{theorem}

\newcommand{\Z}{\mathbb{Z}}
\newcommand{\Q}{\mathbb{Q}}

\begin{document}
\title[On some symmetries of the base $ n $ expansion of $ 1/m $ - 1]{On some symmetries of the base $ n $ expansion of $ 1/m $ : Comments on Artin's Primitive root conjecture}

\author{Kalyan Chakraborty}
\email[Kalyan Chakraborty]{kalychak@ksom.res.in}
\address{ Kerala School of Mathematics, KCSTE,
	Kunnamangalam, Kozhikode, Kerala, 673571, India.}

\author{Krishnarjun Krishnamoorthy}
\email[Krishnarjun K]{krishnarjunk@hri.res.in, krishnarjunmaths@gmail.com}
\address{Harish-Chandra Research 
	Institute, HBNI,
	Chhatnag Road, Jhunsi, Prayagraj 211019, India.}

\maketitle

\begin{abstract}
	Suppose $ m,n\geq 2 $ are co prime integers. We prove certain new symmetries of the base $ n $ representation of $ 1/m $, and in particular characterize the subgroup generated by $ n $ inside $ (\Z/m\Z)^\times $. As an application we give a sufficient condition for a prime $ p $ such that a non square number $ n $ is a primitive root modulo $ p $.
\end{abstract}

\section{Introduction}\label{Section "Introduction"}

For a given integer $ m $, the group $ \Z/m\Z $ is well understood. The structure of $ (\Z/m\Z)^\times $ is a bit more mysterious. For example if $ m $ were a prime, it is rather elementary to show that $ (\Z/m\Z)^\times $ is cyclic but it is still unknown, except for a few trivialities, whether a given integer $ n $ co prime to $ m $ will generate $ (\Z/m\Z)^\times $ or not. Such an integer $ n $ is called a primitive root modulo $ m $. One conjecture along these lines is due to Artin \cite{Pieter Moore}.

\begin{Conjecture}[Artin, Qualitative Version]
	If $ n $ is a natural number which is not equal to $ \{0,1\} $ and not a perfect square, then there are infinitely many primes $ m $ for which $ n $ is a primitive root modulo $ m $.
\end{Conjecture}

Artin's conjecture is known conditionally under the generalized Riemann hypthoesis for certain Dedekind zeta functions due to the work of Hooley \cite{Hooley}. More recent works on Artin's conjecture include that of Heath-Brown \cite{Heath-Brown} who showed that there are at most two primes for which the conjecture is false. For a quick survey of this topic, the reader is pointed to \cite{Pieter Moore}. It is rather classical that the behavior of $ n $ modulo $ m $ is intimately tied to the digits appearing in the base $ n $ expansion of $ 1/m $. This provides us with a concrete way to address the problem of primitivity. Nevertheless these digits turn out to be as mysterious as Artin's conjecture itself. These digits have been studied by many authors and recently a connection has been established between these digits and class numbers of certain quadratic fields \cite{Ram-Tha}. The authors hope to study this connection in a future work.

The aim of this article is twofold. Given $ n,m $ co prime, checking the primitivity of $ n $ modulo $ m $ is quite well understood from an algorithmic stand point. The first aim of this article is to describe a ``reverse process" of this method. The authors hope that this will help in resolving Artin's conjecture as this method is a first step in guessing a suitable prime $ m $ for which we can expect $ n $ to be a primitive root.

The second aim of this paper is to prove the following theorem. See \S \ref{Section "Basic Results"} for definitions. 

\begin{theorem}\label{Theorem 1}
	Suppose $ a $ is a non periodic string modulo $ n $ such that the following hold.
	\begin{enumerate}
		\item	There exists a prime $ p $ such that $ (p-1) | \mathcal{O}_a(n) $ and $ p | m $,
		\item	For every other prime $ q | m$, we have $ \mathcal{O}_q(n) | \mathcal{O}_p(n) $,
	\end{enumerate}
	where we have set 
	\[
	n^{\mathcal{O}_a(m)}-1 = a\cdot m.
	\]
	Then $ n $ is a primitive root modulo $ p $.
\end{theorem}

Theorem \ref{Theorem 1} provides a sufficient condition for $ n $ to be a primitive root modulo $ p $ and also hints at where to look for such primes.

The paper is structured as follows. In \S \ref{Section "Basic Results"} we collect some basic results regarding the digits of $ 1/m $ in base $ n $. We also provide proofs, firstly to fix the notation, and secondly because these results although elementary, have not been proved when $ m $ is not a prime, to the best of authors' knowledge. In \S \ref{Section "Reverse Algorithm"} we describe how the ``reverse process" mentioned before. Finally in \S \ref{Section "Artin Primitive root"} we prove Theorem \ref{Theorem 1}.

\section*{Acknowledgements}

The author would like to thank the Kerala School of Mathematics for its generous hospitality.

\section{Base $ n $ representations and Primitive roots}\label{Section "Basic Results"}

Consider an integer $ n\geq 2 $ which shall remain fixed for the rest of the section. Suppose that $ m $ is another natural number such that $ (n,m)=1 $. Consider the base $ n $ representation of $ 1/m $,
\begin{equation}\label{Equation "Base n representation 1/m"}
	\frac{1}{m} = \sum_{k=1}^{\infty} \frac{a_k(m)}{n^k}.
\end{equation}
The digits $ a_k(m) $ each satisfy $ 0\leq a_k(m)<n $ and are unique for a given $ m $. Furthermore they are periodic in $ k $ and the period equals the order of $ n $ in $ (\Z/m\Z)^\times $. 

The digits $ a_k(m) $ satisfy lot of additional symmetries. To describe them, suppose that $ 0\leq a_1,\ldots,a_l < n $ are integers. The ordered tuple $ (a_1,\ldots,a_l) $ is then called a \textit{string of length} $ l $ modulo $ n $. Given a string $ (a_1,\ldots,a_l) $ modulo $ n $, define
\begin{align}\label{Equation "String integer definition"}
	 (a_1,\ldots,a_l)_n &:=\frac{1}{n}\sum_{k=1}^l a_{l-k+1} n^k,\\
	 \overline{(a_1,\ldots,a_l)} & := \sum_{k=1}^{\infty} \frac{a_k}{n^k},
\end{align}
where, for $ k >l $ we take $ a_k := a_{k-l} $. Finally we say that the string $ (a_1,\ldots,a_l) $ is associated to an integer $ m $ if $ (a_1,\ldots,a_l) $ is a string of smallest length such that $ 1/m = \overline{(a_1,\ldots,a_l)} $. In this case for ease of notation, we shall denote $ (a_1,\ldots,a_l) $ as $ (m) $. This is clearly unique. We shall denote the order of $ n $ in $ (\Z/m\Z)^\times $ as $ \mathcal{O}_m(n) $. Suppose for an integer $ g $ we denote by $ [g]_n $ to be the least positive integer such that $ [g]_n\equiv g\mod n $. Given an integer $ a $ whose base $ n $ representation has $ l $ digits, we will say that $ a $ is periodic with respect to $ n $ with period $ k $ if the base $ n $ digits of $ a $ satisfy $ a_i=a_{i+k} $ for all $ 1\leq i \leq k-l $. Equivalently, we say that $ a $ is periodic with period $ l $ if $ a $ is divisible by 
\[
\frac{n^{k}-1}{n^l-1}.
\]
We observe here that $ \mathcal{O}_a(n) | l $ if $ a $ is periodic with period $ l $.
The symbols $ m,n $ are reserved for mutually prime integers and we shall primarily be trying to understand the subgroup generated by $ n $ inside $ m $.

\begin{lemma}\label{Lemma 1}
	The following relation holds for all $ m $ coprime to $ n $,
	\[
	(m)_n = \frac{n^{\mathcal{O}_m(n) }-1}{m}.
	\]
\end{lemma}

\begin{proof}
	Since we have $ 1/m = \overline{(m)} $, it follows that 
	\[
	\left\lfloor\frac{n^{\mathcal{O}_m(n) }}{m}\right\rfloor = (m)_n\mbox{ and }\left\{\frac{n^{\mathcal{O}_m(n) }}{m}\right\} = \overline{(m)} = \frac{1}{m}.
	\]
	The lemma follows from here. Here we have used the notations $ \lfloor x\rfloor $ and $ \{x\} $ for the integer and fractional parts of $ x $ respectively.
\end{proof}

\begin{lemma}\label{Lemma 2}
	Suppose that $ k $ is the smallest integer such that $ a_k(m)\neq 0 $. Then $ k=\lceil \log_n(m)\rceil $.
\end{lemma}

\begin{proof}
	From the assumption it follows that 
	\[
	\frac{1}{n^{k-1}} > \frac{1}{m} > \frac{1}{n^k}.
	\]
	Now the lemma follows.
\end{proof}

\begin{lemma}\label{Lemma 3}
	Suppose that $ (m) = (a_1,a_2,\ldots,a_l) $. If $ a\equiv n^t \mod m $ for some $ t\geq 0 $, then, 
	\[
	\frac{1}{a} = \overline{(a_{t+1},\ldots, a_l,a_1,\ldots, a_t)}.
	\]
	Conversely, for every $ t\geq 0 $, 
	\[
	\overline{(a_{t+1},\ldots, a_l,a_1,\ldots, a_t)}=\left\{\frac{n^t}{m}\right\}.
	\]
\end{lemma}

The string $ (a_{t+1},\ldots, a_l,a_1,\ldots, a_t) $ is called a cyclic permutation of $ (a_1,a_2,\ldots,a_l) $ and is denoted by $ \sigma_t((a_1,a_2,\ldots,a_l)) $.

\begin{proof}
	The proof follows from the observation that
	\begin{equation}\label{Equation "Fractional part of n^t/m"}
		\left\{\frac{a}{m}\right\} = \left\{\frac{n^t}{m}\right\} = \frac{[n^t]_m}{m}
	\end{equation}
	and from the definition of $ (m) $.
\end{proof}

The proof of the following lemma is an easy modification of proof of Lemma 1 of \cite{Ram-Tha}.

\begin{lemma}\label{Lemma 4}
	With $ n,m $ as above, we have 
	\[
	a_k(m) = \frac{n[n^{k-1}]_m-[n^k]_m}{m}.
	\]
\end{lemma}

\begin{proof}
	We start with
	\begin{align*}
		\sum_{k=1}^\infty \frac{n[n^{k-1}]_m-[n^k]_m}{mn^k}&= \frac{1}{m} \sum_{k=1}^\infty\frac{n[n^{k-1}]_m-[n^k]_m}{n^k}\\
		&= \frac{1}{m}\left(	\sum_{k=1}^\infty \frac{[n^{k-1}]_m}{n^{k-1}} - \sum_{k=1}^\infty\frac{[n^k]_m}{n^k}\right)\\
		&=\frac{1}{m}.
	\end{align*}
Now the lemma follows from the uniqueness of $ a_k(m) $s once we observe that $ \frac{n[n^{k-1}]_m-[n^k]_m}{m} $ are integers lying in between $ 0 $ and $ n-1 $.
\end{proof}

As a particular consequence of the previous lemma, we have that
\begin{equation}\label{Equation "Congruence of ak"}
	ma_k(m)\equiv -[n^k]_m \mod n.
\end{equation}

\begin{lemma}\label{Lemma 5}
	Suppose that $ m $ is a prime and that $ \mathcal{O}_m(n) $ is even for some $ n\not\equiv -1\mod m $. Then for every $ k $ we have 
	\[
	a_k(m) + a_{k+\frac{\mathcal{O}_m(n)}{2}}(m) = n-1.
	\]
\end{lemma}

\begin{proof}
	From \eqref{Equation "Congruence of ak"} we have
	\[
	m(a_k(m)+a_{k+\frac{\mathcal{O}_m(n)}{2}}(m)) \equiv -[n^k]_m - [n^{k+\frac{\mathcal{O}_m(n)}{2}}]_m\mod n.
	\]
	Since $ n^{k+\frac{m-1}{2}}\equiv -n^k\mod m $, we have 
	\begin{align*}
		m(a_k(m)+a_{k+\frac{\mathcal{O}_m(n)}{2}}(m)) &\equiv -[n^k]_m + [n^k]_m -m \mod n\\
		&\equiv -m \mod n.
	\end{align*}
	Since $ (n,m)=1 $, this in particular gives us that 
	\[
	a_k(m)+a_{k+\frac{\mathcal{O}_m(n)}{2}}(m)\equiv -1\mod n.
	\]
	Now since $ 0\leq a_k(m),a_{k+\frac{\mathcal{O}_m(n)}{2}}(m) \leq n-1 $, the claim follows.
\end{proof}

\subsection{The case when $ n=2 $}\label{Subsection "Base 2"}

Further properties of the string $ (m) $ can be obtained. For simplicity suppose that $ n=2 $ is a primitive root modulo $ m > 3 $. Suppose that $ (m) = 0_{t_1}1_{t_2}0_{t_3}\ldots 1_{t_l} $. Here $ 0_j $ stands for $ \underbrace{00\ldots 0}_{\tiny{j \mbox{ times}}} $.

From Lemma \ref{Lemma 5} it follows that $ l\equiv 2\mod 4 $ and that $ t_i = t_{l/2+i} $. From Lemma \ref{Lemma 3}, the cyclic permutations of $ \overline{(m)} $ give rise to $ a/m $ for every $ 1\leq a\leq m-1 $. Therefore there exists a $ s $ such that $ \overline{\sigma_s(0_{t_1},\ldots,1_{t_l})} = 3/m $. In binary digits $ 3=11 $ and therefore multiplication by $ 3 $ equals multiplication by $ 11 $. From Lemma \ref{Lemma 2} it is clear that $ t_1 > 0 $ and by construction $ t_3 > 0 $. For simplicity, suppose that $ t_2 > 2 $. Consider the following sum,
\begin{equation*}
	0\overbrace{11\ldots11}^{\tiny{t_2\mbox{ times}}}+\overbrace{11\ldots11}^{\tiny{t_2\mbox{ times}}}0 = 10 \overbrace{11\ldots11}^{\tiny{t_2-2\mbox{ times}}}01.
\end{equation*}
Therefore, it follows that for every even $ i $, if $ t_i > 2 $ then there exists an even $ j $ such that $ t_j=t_i-2 $.

\section{A Reconstruction Algorithm}\label{Section "Reverse Algorithm"}

If $ m $ is co prime to $ n $, we have seen in the previous section that 
\begin{equation}\label{Equation "m (m) connection"}
	(m)_n m = n^{\mathcal{O}_m(n)}-1.
\end{equation}
Therefore it follows that $ \mathcal{O}_m(n) $ is divisible by the order of $ n $ in $ (\Z/(m)_n\Z)^\times $. In this section we try to answer the question given a string of integers $ (a_1,\ldots,a_l) $ modulo $ n $, what is the corresponding $ m $ and when will that $ m $ be a prime.

For simplicity suppose we denote $ (a_1,\ldots,a_l)_n = a $. Since we have assumed that $ a_l $ is co prime to $ n $, it follows that $ a $ is coprime to $ n $. The previous section tells us that our candidate for $ m $ should be $ \overline{(a_1,\ldots,a_l)}^{-1} $. Now, there is a problem we need to address here. The value $ a $ does not change if we add zeroes to the left of the string $ (a_1,\ldots,a_l)_n $ but doing so will change the value of $ \overline{(a_1,\ldots,a_l)} $. Therefore we make the following convention. Clearly $ \mathcal{O}_a(n) > l-1 $. Now we add zeroes to the left of the string $ (a_1,\ldots,a_l) $ so that $ \mathcal{O}_a(n) = l $. By abuse of notation, we shall denote the string after adding zeroes also by $ (a_1,\ldots,a_l) $. Now define $ m=\overline{(a_1,\ldots,a_l)}^{-1} $, or in other words 
\begin{equation}\label{Equation "m definition"}
	m:=\frac{n^{\mathcal{O}_a(n)}-1}{a}.
\end{equation}
Clearly $ m $ divides $ n^{\mathcal{O}_a(n)}-1 $ and therefore $ \mathcal{O}_m(n) | \mathcal{O}_a(n) $. We also see from Lemma \ref{Lemma 3} that the subgroup generated by $ n $ inside $ (\Z/m\Z)^\times $ is given by cyclic permutations of the string $ (a_1,\ldots,a_l) $.

Furthermore, if we assume that $ (a_1,\ldots,a_l) $ has no repetitions, we see that $ \mathcal{O}_m(n) = l = \mathcal{O}_a(n) $.

\section{Comments on Artin's primitive root conjecture}\label{Section "Artin Primitive root"}

Before we prove Theorem \ref{Theorem 1} we need the following preparatory result.

\begin{proposition}\label{Proposition 1}
	Suppose that $ m $ is an odd prime. A natural number $ n $ is a primitive root modulo $ m $ if and only if the base $ n $ representation of 
	\[
	\alpha:=\frac{n^{m-1}-1}{m}
	\]
	is not periodic.
\end{proposition}

\begin{proof}
	Suppose that $ n $ is a primitive root modulo $ n $, then from Lemma \ref{Lemma 1} it follows that $ \alpha=(m)_n $. Therefore, by definition, the base $ n $ representation of $ \alpha $ is not periodic.
	
	Conversely suppose that the base $ n $ representation of $ \alpha $ is not periodic. It is clear that $ \mathcal{O}_m(n) $ divides $ m-1 $. Set $ (m-1)/\mathcal{O}_m(n) = k $. Then it follows from induction and Lemma \ref{Lemma 1} that
	\[
	 \alpha = (m)_n \left(n^{(k-1)\mathcal{O}_m(n)}+ \ldots + 1\right) = (m)_n \left(\frac{n^{m-1}-1}{n^{\mathcal{O}_m(n)}-1}\right).
	\]
	But since $ \alpha $ is not periodic, it follows that $ k-1=0 $ or equivalently $ m-1=\mathcal{O}_m(n) $. Thus $ n $ is a primitive root.
\end{proof}

\begin{proof}[Proof of Theorem \ref{Theorem 1}]
	Suppose we start with an $ n $ and a non periodic string $ (a_1,\ldots,a_l) $ such that $ \mathcal{O}_a(n) = l $ (with respect to the notation from \S \ref{Section "Reverse Algorithm"}). Suppose we define $ m $ as in \eqref{Equation "m definition"} and that 
\[
m = p^e\prod_{i=1}^{j} q_i^{e_i}
\]
is the prime factorization of $ m $. Furthermore, since 
\[
am = n^{\mathcal{O}_m(n)}-1,
\]
we see that every prime $ q $ satisfies,
\[
q|am \Leftrightarrow \mathcal{O}_q(n) | \mathcal{O}_m(n).
\]
Now, from the chinese remainder theorem, there is an isomorphism
\[
\left(\Z/m\Z\right)^\times \leftrightarrow (\Z/p^e\Z)^\times \times \prod_{i=1}^{j} \left(\Z/q_i^{e_i}\Z\right)^\times.
\]
Dividing out $ m $ by certain prime powers if necessary, we can suppose that $ n\not\equiv 1 \mod q_i^{e_i} $ for every $ i $.
Furthermore, it also follows that 
\[
\mathcal{O}_m(n) = LCM(\mathcal{O}_{p^e}(n),\mathcal{O}_{q_i^{e_i}}(n)) = \mathcal{O}_{p^e}(n).
\]
Therefore $ (p-1) $ divides $ \mathcal{O}_{p^e}(n) $. In particular we have $ \mathcal{O}_p(n)=p-1 $. This completes the proof.
\end{proof}

\section{Concluding Remarks}\label{Section "Concluding Remarks"}

Suppose we have a non square positive integer $ n >1 $ and we want to ``guess'' a prime $ m $ such that $ n $ is a primitive root modulo $ m $. Theorem \ref{Theorem 1} provides a first order filter in the form of the non periodic sequence $ (a) $. Although it is possible to choose various candidates for $ (a) $, the primes occurring in the assumptions of Theorem \ref{Theorem 1} are the divisors of $ a $ along with $ m $. Or in other words primes $ q $ such that $ \mathcal{O}_q(n) $ divides $ \mathcal{O}_a(n) $. By restricting the choice of $ a $, it should be possible to exercise some control over the primes dividing $ am $ so that the conditions of Theorem \ref{Theorem 1} can be met. But it is not entirely clear how this can be done for an infinite family. Also, while constructing such an infinite family of strings it should be kept in mind that the primes $ p $ (in the notation of Theorem \ref{Theorem 1}) should be distinct.

It would be very convenient if the $ m $ obtained from $ (a) $ themselves turn out to be prime. Being able to construct such strings would amount to completely understanding the symmetries under cyclic permutations mentioned in Lemma \ref{Lemma 3}. Furthermore, the preliminary lemmas proved in \S \ref{Section "Basic Results"} and \S \ref{Subsection "Base 2"} serve as necessary conditions for a string $ (a) $ to correspond to a prime $ m $. It would be interesting to characterize the strings which are connected to prime $ m $.

\end{document}